\documentclass[11pt,reqno]{amsart}
\usepackage{amsmath,amsfonts,amssymb,amscd,amsthm,amsbsy,bbm, epsf,calc, color}
\usepackage{datetime}
\usepackage{lineno}

\newcounter{hours}\newcounter{minutes}

\newtheorem{thm}{Theorem}[section]
\newtheorem{lem}[thm]{Lemma}

\newtheorem{prop}[thm]{Proposition}
\newtheorem{conj}[thm]{Conjecture}

\newtheorem{DEF}[thm]{Definition}
\theoremstyle{remark}                  
\newtheorem{rem}[thm]{Remark}

\def\F{{\mathcal F}}

\def\P{{\mathbb P}}

\def\real{{\mathbb R}}
\def\rational{{\mathbb Q}}
\def\Indicator{{\mathbbm{1}}}

\def\ep{\varepsilon}
\def\al{\alpha}
\def\del{\delta}
\def\om{\omega}
\def\Om{\Omega}
\def\gam{\gamma} 
\def\lam{\lambda}
\def\Lam{\Lambda}

\def\sig{\sigma}

\def\diam{\textnormal{diam}}

\def\Tr{\textnormal{Tr}}
\def\Id{\textnormal{Id}}

\def\union{\bigcup}
\def\intersect{\bigcap}

\newcommand{\abs}[1]{\left| #1 \right|}

\newcommand{\ds}{\displaystyle}

\newcommand{\norm}[1]{\lVert#1\rVert}
\newcommand{\inner}[2]{#1 \cdot #2}

\begin{document}

\title{Stochastic Homogenization for Some Nonlinear Integro-Differential Equations}
\author{Russell W. Schwab}
\begin{abstract}
In this note we prove the stochastic homogenization for a large class of fully nonlinear elliptic integro-differential equations in stationary ergodic random environments.  Such equations include, but are not limited to Bellman equations and the Isaacs equations for the control and differential games of some pure jump processes in a random, rapidly varying environment. The translation invariant and non-random effective equation is identified, and the almost everywhere in $\om$, uniform in $x$ convergence of the family solutions of the original equations is obtained.   Even in the linear case of the equations contained herein the results appear to be new.
\end{abstract}
\address{Department of Mathematical Sciences; 6113 Wean Hall; Carnegie Mellon University; Pittsburgh, PA 15213}
\email{rschwab@andrew.cmu.edu}
\date{\today}
\thanks{RWS was partially supported by a NSF Postdoctoral Research Fellowship, grant DMS-0903064.}
\keywords{Homogenization, Levy Processes, Jump Processes, Nonlocal Elliptic Equations, Obstacle Problem, Optimal Stochastic Control }
\subjclass[2000]{35J99, 
45J05, 
47G20, 
49L25, 
49N70, 
60J75, 
93E20 
}

\maketitle
\baselineskip=14pt
\pagestyle{headings}		
\markboth{Russell W. Schwab}{Stochastic Integro-Differential Homogenization}

\section{Introduction And Main Result}\label{sec:Intro}
\setcounter{equation}{0}

\subsection{A brief Introduction}
In this note we present the homogenization for viscosity solutions of a stochastic family of nonlinear, integro-differential equations given by
\begin{equation}\label{eq:PIDEmain}
\begin{cases}
\ds F(u^\ep,\frac{x}{\ep},\om)=0 &\text{ in } D\\
u^\ep=g &\text{ on } \real^n\setminus D,
\end{cases}
\end{equation}
where $D$ is an open, bounded domain in $\real^n$ and $\om\in\Om$ for some probability space, $(\Om,\F,\P)$.
In this context the operator, $F$, will take the form:
\begin{equation}\label{eq:FeqFormatIntro}
\begin{split}
&F(u,\frac{x}{\ep},\om) = \\
&\ \ \inf_{\al}\sup_{\beta} \big\{ f^{\al\beta}(\frac{x}{\ep},\om) + \int_{\real^n} (u(x+y)+u(x-y)-2u(x))K^{\al\beta}(\frac{x}{\ep},y,\om)dy\big\},
\end{split}
\end{equation}
with $K^{\al\beta}(x,y,\om)$ satisfying particular assumptions below.  Here the coefficients of the equation, appearing as the kernels $K^{\al\beta}(x,y,\om)$ as well as the forcing terms $f^{\al\beta}(x,\om)$, form a stationary ergodic family (in the variable, $x$) for $\om\in\Om$, with respect to an ergodic group of transformations, $\tau_x:\Om\to\Om$ (elaborated below, in (\ref{eq:AssumptionStationaryK})-(\ref{eq:AssumptionErgodic})).

Such operators appear as the infinitesimal generators of pure jump processes (and the generators of the corresponding optimal control problems and differential games-- see \cite{Chan-99, Hans-07, OkSu-07, OkSu-2009RiskIndifferenceJumpDiffusion, Pham-98} and the references therein for \emph{similar} operators to (\ref{eq:FeqFormatIntro}) in a deterministic environment in the context of, e.g. Mathematical Finance) in a rescaled random media with non-homogeneous jump distributions randomly given as $K^{\al\beta}(x/\ep,y,\om)dy$.  Roughly speaking, it is expected that the high frequency oscillations of the stationary ergodic family of equations, modeled by the scaling $K^{\al\beta}(x/\ep,y,\om)$, will lead to an averaging property of the solutions of (\ref{eq:PIDEmain}), e.g. they have a limiting behavior towards a translation invariant equation.

Recently, the homogenization for this class of integro-differential equations with rapidly oscillating \emph{periodic} coefficients was proved in \cite{Schw-10Per}.  There it was demonstrated, see \cite[Remark 3.6]{Schw-10Per}, that those methods would generalize to the random setting, if certain Aleksandrov-Bakelman-Pucci type estimates would hold for the equations under consideration.  

We briefly elaborate on this comment.  For the sake of explanation, consider the linear case of (\ref{eq:PIDEmain}) with $L$ as an integro-differential operator given by
\begin{equation}
L(u,x)=\int_{\real^n}\left(u(x+y)+u(x-y)-2u(x)\right)K(x,y)dy,
\end{equation}
and $u_k$ as subsolutions of 
\begin{equation*}
\begin{cases}
\ds L(u_k,x)\geq -g_k(x) &\text{in}\ B\\
u_k\leq0 &\text{on}\ \real^n\setminus B,
\end{cases}
\end{equation*}
where for simplicity we assume $g_k\geq0$.  Since the constant, $0$, is also a subsolution, we may without loss of generality replace $u_k$ by $u_k^+=\max\{u_k,0\}$, and so we assume without loss of generality that $u_k\geq0$.

If for example $g_k\leq1$ (these choices will appear more relevant in the context of Lemma \ref{lem:ZeroContactLimit}) and 
\begin{equation*}
\abs{\{x:g_k(x)>0\}}\to0\ \text{as}\ k\to\infty,
\end{equation*}
then does it follow that $\norm{u_k}_{L^\infty}\to0$ as well?  We make this question for the nonlinear analogue of this situation precise in Proposition \ref{prop:MminusRHSToZero} below, and refer to it as a ``comparison with measurable ingredients'' (cf. \cite{CaCrKoSw-96} or \cite[Chapter 9]{GiTr-98} for second order equations).  Its specific use for homogenization can be seen in Section \ref{sec:SubadditiveAtX0}.  An affirmative answer to this question was noted in \cite[Remark 3.6 and Section 6]{Schw-10Per} as being simultaneously both fundamental to making the methods of homogenization for second order equations in \cite{CaSoWa-05} apply to the fractional order setting and also as a basic result in the analysis of integro-differential which is missing from the current literature.

Some partial results related to ``comparison with measurable ingredients'' have been presented in \cite[Theorem 1.3]{GuSc-12ABParma}, and hence provide a stimulus for the current investigation.  In this work we extend the homogenization results of the periodic case in \cite{Schw-10Per} to the general stationary ergodic case given by (\ref{eq:FeqFormatIntro}).  These methods are generic from the point of view of homogenization, and are simply dependent upon a ``comparison with measurable ingredients'' result to hold true within the corresponding class of elliptic equations.  This is strongly believed to hold in very general settings including the particular one of (\ref{eq:FeqFormatIntro}).  This work draws upon many of the techniques and results built up in \cite{Schw-10Per}, and so it may be considered as a sequel to \cite{Schw-10Per}.

For a general introduction to homogenization, the curious reader should consult the books \cite{BeLiPa-78} and \cite{Jiko-94}, and we will give a more complete list regarding stochastic homogenization in Section \ref{sec:Background}.  For definitions and \emph{basic} results for viscosity solutions of equations related to and or including (\ref{eq:PIDEmain}) and (\ref{eq:PIDEaveraged}), the reader should consult \cite{BaChIm-08Dirichlet}, \cite{BaIm-07}, and \cite[Sections 1-5]{CaSi-09RegularityIntegroDiff}.  For viscosity solutions in the context of first and second order equations the reader should consult \cite{CrIsLi-92}.

\subsection{Main Theorems and Propositions}
The results we prove will show the existence of an effective nonlocal equation such that the family of solutions governed by (\ref{eq:PIDEmain}) converges locally uniformly to the solution of this effective equation.  The important features are that the effective equation is nonlocal, elliptic, and translation invariant, given by
\begin{equation}\label{eq:PIDEaveraged}
\begin{cases}
\ds \Bar F(\Bar u,x)=0 &\text{ in } D\\
\Bar u=g &\text{ on } \real^n\setminus D.
\end{cases}
\end{equation}
This behavior of $u^\ep$ is described in the main theorems of the note:

\begin{thm}\label{thm:main}
Assume (\ref{eq:AssumptionStationaryK})-(\ref{eq:AssumptionfBounded}), (\ref{eq:AssumptionKernelsGS}), and that uniqueness holds for viscosity solutions of (\ref{eq:PIDEmain}). Then there exists a set of full measure, $\Tilde\Om$, and a translation invariant operator, $\Bar F$, which describes a nonlocal ``elliptic'' equation such that for all $\om\in\Tilde\Om$ and any choice of uniformly continuous data, $g$, the solutions of (\ref{eq:PIDEmain}) converge locally uniformly to the unique $\bar u$ which solves (\ref{eq:PIDEaveraged}).  Moreover $\Bar F$ is ``elliptic'' with respect to the same extremal operators as the original operator, $F$, given in (\ref{eq:MminusDef}) and (\ref{eq:MPlusDef}).  
\end{thm}

\begin{thm}\label{thm:general}
Assume (\ref{eq:AssumptionStationaryK})-(\ref{eq:AssumptionEllipticityPointwise}), that uniqueness holds for viscosity solutions of (\ref{eq:PIDEmain}), and that Conjecture \ref{con:ComparisonMeasurable} is true.  Then the same outcome of Theorem \ref{thm:main} holds true. 
\end{thm}

\begin{rem}
In most stochastic homogenization results, the final set of full measure, $\ds\Tilde\Omega$, on which the convergence happens is an intersection of many auxiliary sets appears along the way in the proof.  For the curious reader, we give a reasonably detailed accounting of the origins of $\ds\Tilde\Omega$ in Remark \ref{rem:OmegaTilde}.
\end{rem}

\begin{rem}\label{rem:HomogVsUniqueness}
In this work we are concerned with proving the homogenization of (\ref{eq:PIDEmain}), and therefore assume that the particular $F$ does indeed admit unique solutions.  The current understanding for uniqueness of (\ref{eq:PIDEmain}) is still incomplete, and we do not focus on the myriad of different assumptions which ensure unique solutions.  Examples of some operators which do admit unique solutions to (\ref{eq:PIDEmain}) were presented in \cite{Schw-10Per}.
\end{rem}

\begin{rem}
In the uniformly elliptic and Hamilton-Jacobi contexts, it is helpful to think of homogenization very loosely as an outcome which is enforced by the solutions' balance of the simultaneous behavior of high frequency oscillations due to the coefficients of the equation and the a priori regularity results imposed by the uniform ellipticity of the equation (or uniform coercivity in the case of Hamilton-Jacobi).  Therefore, it is natural to see the most important assumptions of Theorems \ref{thm:main} and \ref{thm:general} to be aligned with those of stationary ergodicity (oscillations) and regularity (uniform ellipticity) as opposed to assumptions related to uniqueness.
\end{rem}

\begin{rem}
The interested reader should consult \cite[Sections 3-5]{CaSi-09RegularityIntegroDiff} or \cite{BaIm-07} for the basic definitions and properties of viscosity solutions for (\ref{eq:PIDEmain}) and (\ref{eq:PIDEaveraged}).  For a general elliptic nonlocal operator, we use the notion of \cite[Definition 3.1]{CaSi-09RegularityIntegroDiff} for ellipticity. 
\end{rem}

The heart of the homogenization result lies in what is referred to as the solution of the ``corrector'' equation.  This proposition is the main difficulty in proving Theorems \ref{thm:main} and \ref{thm:general} in stationary ergodic environments.  We record it here, and expand upon its motivation and notation below in Section \ref{sec:BackgroundSubSecCorrector}.  All of Sections \ref{sec:SubadditiveAtX0} and \ref{sec:Corrector} are dedicated to its proof.

\begin{prop}[Solving The ``Corrector'' Equation]\label{prop:Corrector}
Assume that the hypotheses of either Theorem \ref{thm:main} or \ref{thm:general} are satisfied.  Let $\phi\in C^{1,1}(\real^n)\intersect L^{\infty}(\real^n)$.  Define the frozen operator at $\phi$ and $x_0$ using (\ref{eq:FrozenOperatorDefLinear}) as:
\begin{equation}\label{eq:FrozenOperatorDef}
F_{\phi,x_0}(v,x,\om):= \inf_{\al}\sup_\beta\left\{f^{\al\beta}(x,\om) + [L^{\al\beta}(\om)\phi(x_0)](x) + [L^{\al\beta}(\om)v(x)](x)\right\}.
\end{equation}
Then there exists a \textbf{unique} number, $\Bar F(\phi,x_0)$, and a set of full measure $\Om_{\phi}\subset\Om$, such that for $\om\in\Om_\phi$ the unique solutions, $v^\ep(\om)$, of
\begin{equation}\label{eq:CorrectorEqDef}
\begin{cases}
\ds F_{\phi,x_0}(v^\ep,\frac{x}{\ep},\om)=\Bar F(\phi,x_0) &\text{ in } B_1(x_0)\\
v^\ep=0 &\text{ on } \real^n\setminus B_1(x_0),
\end{cases}
\end{equation}
also satisfy the \textbf{decay property}
\begin{equation}
\norm{v^\ep}_{L^\infty}\to0\ \  \text{as}\ \ \ep\to0.
\end{equation}
\end{prop}

The main technical lemma which allows for the leap from the periodic to the stationary ergodic settings in the proof of Proposition \ref{prop:Corrector} is the ``comparison with measurable ingredients'' result.  It is a direct corollary of the Aleksandrov-Bakelman-Pucci type estimate recently proved in \cite[Theorem 1.3]{GuSc-12ABParma} (also listed here in Section \ref{sec:Appendix} for convenience).

\begin{prop}[Comparison With Measurable Ingredients]\label{prop:MminusRHSToZero}
Assume that (\ref{eq:AssumptionKernelsGS}) holds.  Suppose that $g_\ep\in C(B)$, $\norm{g_\ep}_{L^\infty}\leq C$, and the sequence $\{v_\ep\}$ solves in the viscosity sense
\begin{equation}\label{eq:ComparisonMeasIngr}
\begin{cases}
\ds M_A^+(v_\ep,x)\geq -g_\ep(x) &\text{in}\ B\\
v_\ep\leq0 &\text{on}\ \real^n\setminus B,
\end{cases}
\end{equation}
where $M^+_A$ is defined in (\ref{eq:MplusSpecialA}).  If $\abs{\{g_\ep>0\}}\to0$ as $\ep\to0$, then $\norm{v_\ep}\to0$ as $\ep\to0$.
\end{prop}

It is widely expected that Proposition \ref{prop:MminusRHSToZero} holds in much more general circumstances, but to date has only been proved in the setting mentioned above.  We therefore include these more general circumstances in Theorem \ref{thm:general} and list the needed comparison result here as a conjecture.

\begin{conj}[General Comparison With Measurable Ingredients]\label{con:ComparisonMeasurable}
Assume (\ref{eq:AssumptionSymmetry}) and (\ref{eq:AssumptionEllipticityPointwise}), then the outcome of Proposition \ref{prop:MminusRHSToZero} holds true with the operator $M^+_A$ replaced by $M^+_{CS}$  (defined in (\ref{eq:MPlusDef})), which is the appropriate extremal operator for (\ref{eq:AssumptionEllipticityPointwise}).
\end{conj}

\subsection{Organization of The Paper} 
It is worth commenting on the presentation of the proofs of Theorems \ref{thm:main} and \ref{thm:general}.  In fact, as soon as either Proposition \ref{prop:MminusRHSToZero} or Conjecture \ref{con:ComparisonMeasurable} hold true, there is no difference in the proof of Proposition \ref{prop:Corrector} and hence also the two main theorems.  For this reason, we have chosen to present the proof of Proposition \ref{prop:Corrector} in the most general setting.  In this case, the reader can appropriately substitute the particular extremal operators, $M^{-+}_A$ or $M^{-+}_{CS}$, for $M^{-+}$ in the remainder of the note.  The only difference being in the former, one is operating under Proposition \ref{prop:MminusRHSToZero} and (\ref{eq:AssumptionKernelsGS}) which is known to be true, and in the latter one is operating under Conjecture \ref{con:ComparisonMeasurable} and (\ref{eq:AssumptionEllipticityPointwise}).  The proof of Proposition \ref{prop:Corrector} in the general case is the content of Sections \ref{sec:SubadditiveAtX0} and \ref{sec:Corrector}.  Section \ref{sec:EffectiveAndProof} uses Proposition \ref{prop:Corrector} to complete the proofs of Theorems \ref{thm:main} and \ref{thm:general}.  The Appendix, Section \ref{sec:Appendix} is used to collect helpful background results required for the rest of the paper.

\subsection{Assumptions}\label{sec:AssumptionsForHomg}
Here we list the assumptions on $F$.

\textbf{\underline{Stationary Ergodic:}}
$K^{\al\beta}(x,y,\om):\real^n\times\real^n\times\Om\to \real$ is stationary if
\begin{align}
&(\Om,\F,\P)\ \text{is a probability space}\nonumber\\
&\text{there is a group of measure preserving transformations,}\ \tau_x:\Om\to\Om\ \text{, for}\ x\in\real^n,\nonumber \\
&\text{and}\ K^{\al\beta}\ \text{satisfies the translation relationship},\ K^{\al\beta}(x+z,y,\om)= K^{\al\beta}(x,y,\tau_z\om).\label{eq:AssumptionStationaryK}\\
\intertext{Similarly we use this definition for}
&f^{\al\beta}:\real^n\times\Omega\to\real,\ \ f^{\al\beta}(x+z,\om)=f^{\al\beta}(x,\tau_z\om)\label{eq:AssumptionStationaryf}
\end{align}

Further, the family is stationary ergodic if it is stationary and also the group $\tau_x$ acts on $\Om$ ergodically in the sense that the only invariant sets of $\tau$ are either trivial or full measure, i.e. the following assertion holds
\begin{align}\label{eq:AssumptionErgodic}
&\text{if for all}\ z,\ \tau_z^{-1}E\subset E,\ \text{then}\ \P(E)=0\ \text{or}\ \P(E)=1.
\end{align}

\textbf{\underline{Boundedness of $f^{\al\beta}$:}} It is important that $F(0,x,\om)$ is uniformly bounded, and so we assume
\begin{equation}\label{eq:AssumptionfBounded}
\norm{f^{\al\beta}(\cdot,\om)}_{L^\infty}\leq C \ \ \forall\ \al,\beta,\om.
\end{equation}

\textbf{\underline{Scaling:}} In order that the rescaling $\ep^\sig u(\cdot/\ep)$ maps solutions between domains of size $1/\ep$ and $1$ with the correctly scaled coefficients, it is necessary to assume that the operator $F$ has an appropriate scaling.  Here we assume the scaling as
\begin{equation}\label{eq:AssumptionScaling}
K^{\al\beta}(x,\lam y,\om) = \lam^{-n-\sig}K^{\al\beta}(x,y,\om).
\end{equation}

\textbf{\underline{Ellipticity:}} Going along with ellipticity, there is also an assumption of symmetry for the kernels -- it is simply to allow us to work with operators which do not have a drift.  This assumption appeared in \cite[Section 2]{CaSi-09RegularityIntegroDiff} concerning the related regularity theory for e.g. (\ref{eq:PIDEmain}). This requirement is
\begin{equation}\label{eq:AssumptionSymmetry}
K^{\al\beta}(x,-y)=K^{\al\beta}(x,y).
\end{equation}
The notion of ellipticity comes from extremal operators which control the difference of the operator evaluated on two different functions, as in \cite[Section 3]{CaSi-09RegularityIntegroDiff}.  It says that there are concave, respectively convex, extremal operators, $M^-$, respectively $M^+$, such that
\begin{equation*}
M^-(u-v,x)\leq F(u,x)-F(v,x)\leq M^+(u-v,x).
\end{equation*}

Here we present two of the main classes of elliptic operators:\   

\underline{Ellipticity requirement 1:} one family is those operators, treated in \cite{CaSi-09RegularityIntegroDiff}, which are formed by using kernels that are pointwise comparable to the kernel of the fractional Laplacian,
\begin{equation}\label{eq:AssumptionEllipticityPointwise}
\lam\abs{y}^{-n-\sig}\leq K^{\al\beta}(x,y)\leq \Lam\abs{y}^{-n-\sig}.
\end{equation}
This family has extremal operators given as
\begin{equation}\label{eq:MminusDef}
M_{CS}^-(u,x) = \inf_{\lam\abs{y}^{-n-\sig}\leq K^{\al\beta}(y)\leq \Lam\abs{y}^{-n-\sig}}\left\{\int_{\real^n}\del u(x,y) K^{\al\beta}(y)dy \right\}
\end{equation}
and
\begin{equation}\label{eq:MPlusDef}
M_{CS}^+(u,x) = \sup_{\lam\abs{y}^{-n-\sig}\leq K^{\al\beta}(y)\leq \Lam\abs{y}^{-n-\sig}}\left\{\int_{\real^n}\del u(x,y) K^{\al\beta}(y)dy \right\}.
\end{equation}

\noindent
We have used the shorthand notation 
\begin{equation}\label{eq:DeltauNotation}
\del u(x,y):= u(x+y)+u(x-y)-2u(x)
\end{equation}
to abbreviate the writing of the integro-differential terms (a convenient consequence of (\ref{eq:AssumptionSymmetry})).

\underline{Ellipticity requirement 2:} the second family, treated in \cite{GuSc-12ABParma}, is smaller than the first, but not completely contained within it.  It consists of kernels which are quadratic modifications of the fractional Laplacian, given as
\begin{align}\label{eq:AssumptionKernelsGS}
&K^{\al\beta}(x,y) = \frac{y^TA^{\al\beta}(x)y}{\abs{y}^{n+\sig+2}}\\ 
&\text{where}\ \ A^{\al\beta}(x)\geq0,\ \Tr(A^{\al\beta})\geq \lam,\ \text{and}\ A^{\al\beta}\leq \Lam\Id.\nonumber
\end{align}

\noindent
This yields the relevant extremal operators as 
\begin{align}
M^-_A(u,x)&=\inf_{\Tr(A)\geq\lam\ \text{and}\ A\leq \Lam\Id}\left\{\int_{\real^n}\del u(x,y)\frac{y^T A y}{\abs{y}^{n+\sig+2}}dy\right\}\label{eq:MminusSpecialA}\\
\intertext{and}
M^+_A(u,x)&=\sup_{\Tr(A)\geq\lam\ \text{and}\ A\leq \Lam\Id}\left\{\int_{\real^n}\del u(x,y)\frac{y^T A y}{\abs{y}^{n+\sig+2}}dy\right\}.\label{eq:MplusSpecialA}
\end{align}

\subsection{Notation}

\begin{enumerate}
\item The second difference operator: $\del u(x,y):= u(x+y)+u(x-y)-2u(x)$.
\item v is $C^{1,1}$ from above at $x$  (respectively from below) \cite[Definition 2.1]{CaSi-09RegularityIntegroDiff} if there exists a radius $r$, a vector $p$, and a constant $M$ such that for all $\abs{y}\leq r$,
\begin{equation*}
v(x+y)-v(x)-p\cdot y \leq M\abs{y}^2\ \ (\text{respectively}\  \geq -M\abs{y}^2).
\end{equation*}
If $v$ is $C^{1,1}$ from above and below at $x$, we say $v\in C^{1,1}(x)$
\item The maximal and minimal operators, $M^-$ and $M^+$ are defined in (\ref{eq:MminusDef}) and (\ref{eq:MPlusDef}) as well as (\ref{eq:MminusSpecialA}) and (\ref{eq:MplusSpecialA}).
\item The half relaxed limits $(u^\ep)^*$ and $(u^\ep)_*$ are
\begin{align*}
(u^\ep)^*(x) = \lim_{\ep\to0}\sup_{\{\del\leq\ep,\ \abs{x-y}\leq\ep\}}u^\del(y);\ \  
(u^\ep)_*(x) = \lim_{\ep\to0}\inf_{\{\del\leq\ep,\ \abs{x-y}\leq\ep\}}u^\del(y).
\end{align*}
\item The contact set of an obstacle problem $K(A)=\{U_A=0\}$ where $U_A$ is defined in (\ref{eq:ObstacleGeneric}).
\item The ball of radius $r$ is $B_r(x)\subset\real^n$, and the cube of radius $r$ is \[Q_r(x)=(x-r/2,x+r/2)^n\subset\real^n\]
\end{enumerate}

\section{Background and Main Ideas}\label{sec:Background}
\setcounter{equation}{0}

\subsection{Background}

Stochastic Homogenization for fully nonlinear equations is an important field, which although is currently not nearly as well  studied as the corresponding one for linear equations, seems to be expanding quickly and gaining interest.  The study of stochastic homogenization of linear equations goes back at least to \cite{Kozl-1979AveragingRandOp}, \cite{PaVa-1979BoundaryValProb}, and \cite{PaVa-1982DiffusionRandomCoeff}, and the case of nonlinear equations to \cite{BeBl-1988ControlDiffRandMed} and \cite{DaMo-86bNonLinStochHomogAndErgodic}.  Similarly, the case of stochastic homogenization is not nearly as well studied as that of the periodic case.   Here we give a list of the related results for \emph{stochastic} homogenization.  The list for the periodic setting is much longer, and we do not attempt at a presentation.  For first order equations and ``viscous'' versions of Hamilton-Jacobi equations with convex nonlinearities (``viscous'' being a second order equations whose limit is a first order equation), there are the works of: \cite{KoReVa-06Homog}, \cite{KoVa-06HomogTimeDep}, \cite{LiSo-03Correctors}, \cite{LiSo-05HomogVisc}, \cite{LiSo-10HomogRevisit}, \cite{NoNo-11RandHomogG}, \cite{RezTar-00}, \cite{Schw-09}, \cite{Soug-99}.  In the realm of nonlinear second order elliptic equations, the results are much fewer with basically \cite{CaMe-09HomogObs}, \cite{CaSo-10HomogRate}, \cite{CaSoWa-05}.  Finally, moving to the nonlocal equations, much less has been done in the stochastic setting.  For homogenization of random obstacle problems for a fractional operator there are the works of \cite{CaMe-08HomogFracObs} and  \cite{Foca-2010AperiodicFractionalObs}.  For homogenization for any equations related to (\ref{eq:PIDEmain}), even for the linear version of $F$, there seems to be no existing literature.

Proving Theorem \ref{thm:main} contains two separate steps.  First, one must identify how to extract the influence on $u^\ep$ of the averaging property of the equation itself (from the stationary ergodic family $f^{\al\beta}$ and $K^{\al\beta}$).  This comes with a good choice for an expansion of $u^\ep$ and the identification of the ``corrector'' equation as the main tool to identify the limit equation for $\bar u$.  Moreover this method must also be compatible with the notion of convergence for the weak solutions, $u^\ep$.  Second, one must actually prove that the ``corrector'' equation has a solution.

In the context of elliptic equations (both Hamilton-Jacobi and second order elliptic equations), the correct expansion to extract the averaging properties of $u^\ep$ has been more or less known since the seminal book \cite[Chapter 1, Section 2]{BeLiPa-78} and was first used in nonlinear elliptic equations in \cite{LiPaVa-88unpublished}.  Correspondingly the operator appearing in the corrector equation for first and second order equations has been known for as long as those references, and possibly even longer.  Two recent developments paved the way for the stochastic homogenization for \emph{nonlocal equations}.  First was the realization that the expansion does not require one $v$ to be rescaled and used for all $\ep$ simultaneously, but rather a whole sequence of $v^\ep$ will suffice (see \cite[Section 1]{LiSo-03Correctors}, also mentioned in \cite[Proposition 7.3]{LiSo-05HomogVisc}, and fundamentally used in \cite[Sections 1 and 3]{CaSoWa-05}).  Second was the observation in \cite[Section 2.1]{Schw-10Per} of how the heuristic expansion for $u^\ep$ identifies the appropriate new operator, $F_{\phi,x_0}$, for the solution of the ``corrector'' equation in the nonlocal periodic setting.  The operator identified there is the same one which is used for the ``corrector'' equation in the current work.  The main contribution of this note is the solution of the ``corrector'' equation in the nonlocal, stationary ergodic setting (Proposition \ref{prop:Corrector}).  The rest of the homogenization result is a very straightforward application or minor modification of the existing techniques.

\subsection{Main Ideas For Proposition \ref{prop:Corrector}}\label{sec:BackgroundSubSecCorrector}

The main ingredient of the proof of Theorems \ref{thm:main} and \ref{thm:general} is the solution to the ``corrector'' equation, Proposition \ref{prop:Corrector}.  A reasonably detailed explanation for this proposition was presented in \cite[Section 2.1]{Schw-10Per} (and is more useful when read together with the discussion of \cite[Section 1]{CaSoWa-05}).  We briefly mention here the key ideas.

The key step in proving Theorem \ref{thm:main} is to decide for any admissible test function, $\phi$, and any $x$ fixed, whether or not $\phi$ satisfies 
\begin{align*}
\Bar F(\phi,x) \geq0\ \ 
\text{or}\ \ 
\Bar F(\phi,x) \leq0.
\end{align*}
(This is simply the statement that we know the set where $\Bar F=0$ over the class of necessary test functions.)  The correct choice of inequality(ies) is enforced by the behavior of $u^\ep$ and $\bar u$, and in particular whether or not $\bar u-\phi$ can have a local maximum or minimum at $x$ (or both).  This information is encoded in $u^\ep$ and (\ref{eq:PIDEmain}).  We can extract it informally with an expansion of $u^\ep$.

Heuristically, the correct ansatz for $u^\ep$ is
\begin{equation}\label{eq:Ansatz}
u^{\ep}(x) = \bar u(x)+ \ep^{\sig}v(\frac{x}{\ep}) + o(\ep^\sig).
\end{equation}
To better recognize the two scales inherent in the operators (local and global variables) when using the expansion (\ref{eq:Ansatz}), we rewrite the integro-differential terms (for a \emph{generic $\phi$}) as
\begin{equation}\label{eq:FrozenOperatorDefLinear}
[L^{\al\beta}(\om)\phi(z)](x) = \int_{\real^n} (\phi(z+y)+\phi(z-y)-2\phi(z))K^{\al\beta}(x,y,\om)dy,
\end{equation}
where $z$ is the location of the center of the second difference, and $x$ is the variable in the coefficients, $K^{\al\beta}(x,y,\om)$ (one should note that using $z=x$ in $[L^{\al\beta}(\om)\phi(x)](x)$ gives back the expressions in (\ref{eq:FeqFormatIntro})).  Plugging in (\ref{eq:Ansatz}) into, e.g., the linear case of (\ref{eq:PIDEmain}) with $f^{\al\beta}=0$ and using (\ref{eq:AssumptionScaling}) to scale the integro-differential terms with $\ep^\sig v(\cdot/\ep)$, suppressing the $\om$ from the notation, we obtain
\begin{equation*}
[Lu^\ep(x)](\frac{x}{\ep})=
[L\bar u(x)](\frac{x}{\ep}) + [Lv(\frac{x}{\ep})](\frac{x}{\ep}).
\end{equation*}
If we could possibly find a special function, $v$, which would make this right hand side independent of $\ep$, we would have an equation that reads (thanks also to $[Lu^\ep(x)](x/\ep)=0$)
\begin{equation*}
\Bar F(\bar u,x) = 0,
\end{equation*}
where $\Bar F(\bar u,x)=[L\bar u(x)](x/\ep) + [Lv(x/\ep)](x/\ep)$.  Although this is not exactly possible, we can push the motivation a little further.  As $\ep\to0$, we see that $x$ can be considered a fixed parameter, and the true variable of interest is $y=x/\ep$.  Then we see, we are looking to find a particular $v$ (as a function of $y$!) such that 
\begin{equation}\label{eq:CorrectorMotivation}
[L\bar u(x)](y) + [Lv(y)](y)=\ds constant.
\end{equation}
In order for this $v$ to be useful to (\ref{eq:Ansatz}), we need the compatibility condition that
\begin{equation*}
\norm{\ep^\sig v(\cdot/\ep)}_{L^\infty}\to0\ \text{as}\ \ep\to0,
\end{equation*}
and so it will be the case that not just any constant in (\ref{eq:CorrectorMotivation}) will work.  Finally, it turns out that it is completely unnecessary to require that there is one function $v$, such that $\ep^\sig v(\cdot/\ep)$ works as the correction to the function $\bar u$ at all $\ep$ scales.  This can be replaced by a more generic, $v^\ep$, and the compatibility condition correspondingly reads as
\begin{equation*}
\norm{v^{\ep}}_{L^\infty}\to0\ \text{as}\ \ep\to0.
\end{equation*}
Hence we have arrived at the statement of Proposition \ref{prop:Corrector}.

The main idea to solving the appropriate ``corrector'' equation relevant to (\ref{eq:PIDEmain}) was introduced in \cite[Sections 1 and 3]{CaSoWa-05} for second order equations.  One of the key observations was to view the choice of $\Bar F(\phi,x_0)$ in (\ref{eq:CorrectorEqDef}) as a sort of variational problem on the choice of a constant for the right hand side of (\ref{eq:CorrectorEqDef}).  The second key observation was to introduce a somewhat natural subadditive quantity for (\ref{eq:CorrectorEqDef}), being the measure of the contact set with the obstacle problem in the same domain and with the same operator.

Fixing $x_0=0$, the investigation looks at (\ref{eq:CorrectorEqDef}) for a generic choice of constant right hand side, $l$, given as
\begin{equation}\label{eq:lRHSBackground}
\begin{cases}
\ds F_{\phi,0}(w^{0,\ep}_l,\frac{y}{\ep},\om) = l &\text{ in } B_1(0)\\
w^{\ep}_l(y)=0 &\text{ on } \real^n\setminus B_1(0).
\end{cases}
\end{equation}
The compatibility condition which allows for the correct choice of $l$ (ultimately taken as the $\Bar F$) is the decay of $w^{0,\ep}_l$ -- is it possible to find a particular $l$ so that
\begin{equation*}
\norm{w^{0,\ep}_l}_{L^\infty}\to0\ \text{as}\ \ep\to0?
\end{equation*}
At least for $l$ negative enough, the function 
\begin{equation*}
P^+(x)=\inner{ (1-\abs{x}^2)^2}{\Indicator_{B_1}(x)}
\end{equation*}
will be a subsolution of (\ref{eq:lRHSBackground}).  By comparison, we can conclude that the lower limit of $w^{0,\ep}_l(x)$ will be larger than $P^{+}(x)$.  Thus the compatibility condition was violated in that the limit was too big.  Similarly, for $l$ large enough the function
\begin{equation*}
P^-(x)=\inner{-(1-\abs{x}^2)^2}{\Indicator_{B_1}(x)}
\end{equation*}
is a supersolution of (\ref{eq:lRHSBackground}), and this implies that the upper limit of $w^{0,\ep}_l$ is too negative.  So then there is some hope that with an appropriate choice of $l$, exactly the upper and lower limits of $w^{0,\ep}_l$ can be balanced to give 
\begin{equation*}
\norm{w^{0,\ep}_l}_{L^\infty}\to0\ \text{as}\ \ep\to0.
\end{equation*}
Indeed this is the case, and the correct choice for $l$ is given in Section \ref{sec:Corrector}.  Before one can say which is the correct choice of $l$ above, the effect of generic choice of $l$ on the possible limits of $w^{0,\ep}_l$ must be determined.  This is done using the contact set of an obstacle problem and the Subadditive Theorem in Section \ref{sec:SubadditiveAtX0}.

\section{Subadditive Limits Centered At $x_0=0$}\label{sec:SubadditiveAtX0}
\setcounter{equation}{0}

In preparation for a solution to the ``corrector'' equation of Proposition \ref{prop:Corrector}, this section is dedicated to the investigation of how the choice of $l$ and the ergodicity of $\tau$ affect the solutions, $w^{0,\ep}_l(\om)$, of the equation:

\begin{equation}\label{eq:lRHS}
\begin{cases}
\ds F_{\phi,0}(w^{0,\ep}_l,\frac{y}{\ep},\om) = l &\text{ in } Q_1(0)\\
w^{\ep}_l(y)=0 &\text{ on } \real^n\setminus Q_1(0).
\end{cases}
\end{equation}
(We have switched from $B_1$ to $Q_1$ simply for the convenience of a later analysis involving the Subadditive Theorem, for which the natural choice of sets are cubes.)
The key point is that the limiting behavior of $w^{0,\ep}_l(\om)$ from above and or below can be characterized a.s.$\om$ by applying the Subadditive Ergodic Theorem (found in multiple places, but we refer to \cite{AkKr-81LimitSuperadditive}) to an appropriate quantity related to (\ref{eq:lRHS}).  In applying the Subadditive Theorem, it is important to keep the equations centered at $x_0=0$; this restriction will be expanded upon and relaxed in Section \ref{sec:Corrector}.  

Ultimately we must answer the question given a particular $l$, will $(w^{0,\ep}_l)^*\leq0$ or will $(w^{0,\ep}_l)_*\geq0$?  For the answers to these two questions, we appeal to the fundamental observation to use an appropriate obstacle problem, introduced for homogenization in \cite[Sections 1 and 3]{CaSoWa-05}.  Specifically, we consider the solution of the obstacle problem given in (\ref{eq:ObstacleGeneric}), e.g. as the least supersolution, in the same domain, with the same operator as (\ref{eq:lRHS}), and with an obstacle of the constant 0 function (chosen as $0$ because of the questions of $(w^{0,\ep}_l)^*\leq0$ and $(w^{0,\ep}_l)_*\geq0$). 

The observation is that $w^{0,\ep}_l$ and the obstacle solution will have the same behavior asymptotically if the measure of the contact set between the obstacle solution and the obstacle goes to zero, which is fundamental to Lemma \ref{lem:ZeroContactLimit} and crucially uses Proposition \ref{prop:MminusRHSToZero}.  In particular since the obstacle solution is always above $0$, then $(w^{0,\ep}_l)_*\geq0$ as well.  Furthermore, if the obstacle solution and the obstacle keep positive contact as $\ep\to0$, then since $w^{0,\ep}_l$ is below the obstacle solution, the positive contact with the obstacle forces $(w^{0,\ep}_l)^*\leq0$.  These observations are the key points of this section and appear as Lemmas \ref{lem:ZeroContactLimit} and \ref{lem:PositiveContactLimit}. 
\subsection{The Subadditive Quantity}
Given a bounded domain, $A$, we can solve the obstacle problem with a $0$ obstacle by considering the least supersolution of the equation in (\ref{eq:lRHS}): 
\begin{equation}\label{eq:ObstacleGeneric}
U^l_A(\om)=\inf\big\{u: F_{\phi,0}(u,y,\om)\leq l \text{ in } A \text{ and } u\geq0 \text{ in } \real^n  \big\}.
\end{equation}
It will also be important to solve the same equation, but in a rescaled domain, $\ep A$, with rescaled coefficients, $F(u^\ep,x/\ep)$:
\begin{equation*}
u^{\ep,l}_A(\om)=\inf\big\{u: F_{\phi,0}(u,\frac{y}{\ep},\om)\leq l \text{ in } \ep A \text{ and } u\geq0 \text{ in } \real^n  \big\}.
\end{equation*}
Thanks to (\ref{eq:AssumptionScaling}), the relationship between the two obstacle solutions is
\begin{equation}\label{eq:ScaledAndUnscaledObstacle}
u^{\ep,l}_{A}(x,\om) = \ep^{\sigma}U^l_A(\frac{x}{\ep},\om).
\end{equation}
(Basic properties of the obstacle problem are listed without proof in the Appendix, Section \ref{sec:AppendixObstacleProb}.)  Finally, to connect with $w^{0,\ep}_l$, we make the choice of $A=Q_{1/\ep}(0)$, which gives our obstacle solutions of interest as:
\begin{align}
U^l_{Q_{1/\ep}}(\om)=\inf\big\{u: F_{\phi,0}(u,y,\om)\leq l \text{ in } Q_{1/\ep} \text{ and } u\geq0 \text{ in } \real^n  \big\}  \label{eq:ObstacleBigDomain}\\
\intertext{and}
u^{\ep,l}_{Q_{1}}(\om)=\inf\big\{u: F_{\phi,0}(u,\frac{y}{\ep},\om)\leq l \text{ in } Q_{1} \text{ and } u\geq0 \text{ in } \real^n  \big\}. \label{eq:ObstacleSmallDomain}
\end{align}

The solution of the obstacle problem gives us a very convenient random set function with which to work; namely the measure of the contact set between the solution and the obstacle,
\begin{align}
&M^l(A,\om) := \abs{\{U^l_A(x,\om)=0\}}\label{eq:MlDef}.
\end{align}
For convenience, we will denote the contact set as
\begin{equation}\label{eq:ContactSetNotation}
K^l(A,\om):= \{U^l_A(x,\om)=0\},
\end{equation}
in which case we have
\begin{align}
&M^l(A,\om) := \abs{K^l(A,\om)}.
\end{align}
The main point is that we will be able to use the Subadditive Theorem to extract limits of $M^l(Q_{1/\ep},\om)/\abs{Q_{1/\ep}}$.  Since our original question pertains to $w^{0,\ep}_l$ in $Q_1$, it will be useful to have the analogous quantities to $M^l$ and $K^l$ for $Q_1$:
\begin{align}
&k^{\ep,l}(\om)=\{u^{\ep,l}_{Q_1}(x,\om)=0\}\\
&m^{\ep,l}(\om) := \abs{k^{\ep,l}(\om)}\label{eq:mlEpDef}.
\end{align}
The scaling assumption, (\ref{eq:AssumptionScaling}), tells us the relationship between $U^l_{Q_{1/\ep}}$ and $u^{\ep,l}_{Q_1}$: 
\begin{align}
&m^{\ep,l}(\om) = \frac{1}{\abs{Q_{1/\ep}}}M^l(Q_{1/\ep},\om).\label{eq:ScaledAndUnscaledContactSets}
\end{align}

At this point it is important to recall that the operator in the definition of (\ref{eq:lRHS}) and (\ref{eq:ObstacleBigDomain}), and hence all results derived from them, depend on $\phi$ and $l$.  For right now, $\phi$ and $l$ are fixed parameters, and their use will come up again later, in Section \ref{sec:Corrector}. 

The main point of using the function $M^l(A,\om)$ is that it is stationary and subadditive, presented in the next lemma.

\begin{lem}\label{lem:MStationrySubadditive}
The set function $M^l(A,\om)$ is stationary and subadditive.  Specifically, for $z\in\real^n$ and $A=B_1\cup B_2$ with $B_1$ and $B_2$ having disjoint interiors:
\begin{align*}
&M^l(A+z,\om) = M^l(A,\tau_z\om), \text{ and }\\
&M^l(A,\om)  \leq M^l(B_1,\om)+M(B_2,\om).
\end{align*}
\end{lem}

\begin{proof}[Proof of Lemma \ref{lem:MStationrySubadditive}]
The stationarity is a direct consequence of the translation property of the obstacle solutions, given in Lemma \ref{lem:ObstacleTranslation}, which is simply inherited from the stationarity of the operator $F_{\phi,0}$.  The subadditivity follows from the monotonicity property of the obstacle solutions, given in Lemma \ref{lem:ObstacleMonotonicity}.  Indeed, we have that because $B_1\subset A$ and also $B_2\subset A$, then $U^l_A$ is an admissible supersolution in both of the domains $B_1$ and $B_2$, which gives
\begin{equation*}
U^l_A\geq U^l_{B_1}\ \text{and also}\ U^l_A\geq U^l_{B_2}.
\end{equation*}
Thus 
\begin{equation*}
K^l(A,\om)\intersect B_1\subset K^l(B_1,\om)\ \text{and}\ K^l(A,\om)\intersect B_2\subset K^l(B_2,\om).
\end{equation*}
Hence
\begin{equation*}
\abs{K^l(A,\om)}=\abs{K^l(A,\om)\intersect B_1} + \abs{K^l(A,\om)\intersect B_2}\leq \abs{K^l(B_1,\om)} + \abs{K^l(B_2,\om)}.
\end{equation*}
\end{proof}

The stationarity and subadditivity of $M^l$ allow to use the Subadditive Theorem (see \cite{AkKr-81LimitSuperadditive}) to extract a limit.  This is the content of the next lemma.

\begin{lem}\label{lem:SubAdditiveLimits}
The exists a set of full measure, $\Om_{\phi,l}$, (depending on $\phi$ and $l$) such that the following limits hold for $\om\in\Om_{\phi,l}$: 
\begin{align}
\lim_{\ep\to0}m^{\ep,l}(\om)=\bar m^l(\phi).
\end{align}
\end{lem}

\begin{rem}
It is very important that the limiting quantity depends upon the function, $\phi$, for its use in Section \ref{sec:Corrector}.  However for this current section, $\phi$ is fixed and so we drop the explicit dependence of $\bar m^l$ on $\phi$.
\end{rem}

\begin{proof}[Proof of Lemma \ref{lem:SubAdditiveLimits}]
The Subadditive Theorem (see \cite{AkKr-81LimitSuperadditive}) directly applies to $M^l(Q_{1/\ep},\om)$.  Moreover, the translation group appropriate for the stationarity of $M^l$ is exactly the group, $\tau_x$, from the stationarity of the original equations (\ref{eq:PIDEmain}), which is ergodic (this is not always the case, cf. \cite[Section 4]{Schw-09} where the transformation corresponding to the stationarity of the subadditive quantity was not the original $\tau_x$).  Therefore, there is a \emph{constant}, $\bar m^l$, and a set of full measure, $\Om_{\phi,l}$, such that for $\om\in\Om_{\phi,l}$
\begin{equation*}
\frac{1}{\abs{Q_{1/\ep}}}M^l(Q_{1/\ep},\om)\to \bar m^l.
\end{equation*}
Thus the conclusion of the lemma follows from the relationship between $m^{\ep,l}$ and $M^l$, given in (\ref{eq:ScaledAndUnscaledContactSets}).
\end{proof}

\begin{lem}\label{lem:lBarIncreasing}
$\bar m^l$ is increasing in $l$.
\end{lem}

\begin{proof}[Proof of Lemma \ref{lem:lBarIncreasing}]
Let $l_1\leq l_2$.  By Lemma \ref{lem:ObstacleMonotonicityInRHS}, we know that $U^{l_1}_{Q_{1/\ep}}\geq U^{l_2}_{Q_{1/\ep}}$.  Hence $K^{l_1}(Q_{1/\ep},\om)\subset K^{l_2}(Q_{1/\ep},\om)$.  Taking limits as $\ep\to0$ gives the result.
\end{proof}

\subsection{How The Subadditive Limit Controls The Solution To (\ref{eq:lRHS})}

Now that we know there is a subadditive limit we can extract from $u^{\ep,l}_{Q_1}$ (given as $\bar m^l$), it must still be related back to the behavior of the solutions of (\ref{eq:lRHS}).  The behavior of $w^{0,\ep}_l$ is characterized in Lemmas \ref{lem:ZeroContactLimit} and \ref{lem:PositiveContactLimit}.  At this point, we suppress the explicit dependence upon $\om$ as much as possible.

\begin{lem}\label{lem:ZeroContactLimit}
If $\bar m^l=0$, then $(w^{0,\ep}_l)_*\geq0$.
\end{lem}

\begin{proof}[Proof of Lemma \ref{lem:ZeroContactLimit}]
In this scenario, we will show that the obstacle solution and free solution coincide in the limit.  Therefore, since $u^{\ep,l}\geq0$, we conclude that $(w^{0,\ep}_l)_*\geq0$ as well.

Because $u^{\ep,l}$ is a supersolution of the equation for $w^{0,\ep}_l$, (\ref{eq:lRHS}), we have immediately that $u^{\ep,l}-w^{0,\ep}_l\geq0$.  Therefore, we focus on the reverse inequality.
By definition of elliptic equations, we know that in the viscosity sense
\begin{equation*}
M^+(u^{\ep,l}-w^{0,\ep}_l,x)\geq F_{\phi,x_0}(u^{\ep,l},x)-F_{\phi,x_0}(w^{0,\ep}_l,x). 
\end{equation*}
Owing to (\ref{eq:ObstacleEquationContactRHS}), we have
\begin{equation*}
F_{\phi,x_0}(u^{\ep,l},x)-F_{\phi,x_0}(w^{0,\ep}_l,x)\geq (F_{\phi,x_0}(0,x)-l)\Indicator_{k^\ep_l}(x),
\end{equation*}
and since $F_{\phi,x_0}(0,x)$ is bounded from below (depending on $\phi$), we get the equation for $u^{\ep,l}-w^{0,\ep}_l$:
\begin{equation*}
\begin{cases}
\ds M^+(u^{\ep,l}-w^{0,\ep}_l,x)\geq -C\Indicator_{k^\ep_l}(x) &\text{in}\ Q_1\\
u^{\ep,l}-w^{0,\ep}_l=0 &\text{on}\ \real^n\setminus Q_1.
\end{cases}
\end{equation*}
To apply Proposition \ref{prop:MminusRHSToZero}, we let $g_\ep(x)$ be a continuous approximation of $\Indicator_{k^\ep_l}(x)$ from above.  Thus as $\ep\to0$, $g_\ep$ can be chosen so that $\{g_\ep>0\}\to0$ because we are assuming $k^\ep\to0$.  Therefore, by Proposition \ref{prop:MminusRHSToZero}, $(u^{\ep,l}-w^{0,\ep}_l)^*\leq 0$.  This can be rewritten as
\begin{equation*}
(w^{0,\ep}_l-0)_*\geq (w^{0,\ep}_l-u^{\ep,l})_* \geq 0,
\end{equation*}
which concludes the lemma.
\end{proof}

\noindent
Now we will see which conditions on $\bar m$ imply $(w^{0,\ep}_l)^*\leq0$.
\begin{lem}\label{lem:PositiveContactLimit}
If $\bar m^l>0$, then $(w^{0,\ep}_l)^*\leq0$.
\end{lem}

\begin{proof}[Proof of Lemma \ref{lem:PositiveContactLimit}]
This is a direct consequence of the uniform H\"older regularity of $u^{\ep,l}$ combined with Lemma \ref{lem:SubCubePositiveContact}.  Indeed, we know that $u^{\ep,l}\geq w^{0,\ep}_l$ and given any $r>0$ and $x\in Q_1$ Lemma \ref{lem:SubCubePositiveContact} implies at least one point $\hat x$ with $\abs{x-\hat x}\leq r$ and $u^{\ep,l}(\hat x)=0$.  Therefore
\begin{equation*}
(w^{0,\ep}_l)^*\leq (u^{\ep,l})^*\leq Cr^\gam
\end{equation*}
for a uniform $C$ and $\gam$ corresponding to the regularity of $u^{\ep,l}$ (given by Lemma \ref{lem:ObstacleUniformModulus}).  Since $r$ was arbitrary, we conclude.
\end{proof}

\noindent
The key point used in Lemma \ref{lem:PositiveContactLimit} is the idea that if asymptotically the measure of the contact set is positive, then that positive measure should be spread around $Q_1$ evenly (hence at least one contact point in any subcube of $Q_1$).  This is indeed the case, which is made precise in the next lemma.

\begin{lem}[Positive Contact in Sub-cubes]\label{lem:SubCubePositiveContact}
Assume that $\bar m^l=\al>0$.  For any $r>0$ and any $\eta>0$ fixed, there exists a family of centers, $\{\hat y^\ep_j\}$ and their corresponding cubes, $\{Q_{r+2\rho(\ep)}(\hat y^\ep_j)\}$, such that 
\begin{itemize}
\item[i)] $Q_1(0)\subset \union Q_{r+2\rho(\eta)}(\hat y^\ep_j)$
\item[ii)]  $\abs{\{u^{\ep,l}_{Q_1}=0\}\intersect Q_{r+2\rho(\eta)}(\hat y^\ep_j)}>0$
\item[iii)] $\rho(\eta)\to0$ as $\eta\to0$.
\end{itemize}
\end{lem}

\noindent
The main idea behind Lemma \ref{lem:SubCubePositiveContact} is that the knowledge of the limit of $m^{\ep,l}(\om)$ can be rescaled and translated using $\tau_x$ to any other subcube in $Q_1$.  The problem is that this heuristic is correct only under very careful translations using $\tau$.  Indeed, we note that moving $m^{\ep,l}$ to another cube centered at e.g. $y$ corresponds to looking at $M^l(Q_{r/\ep}(y/\ep),\om)$, and hence by the stationarity $M^l(Q_{r/\ep}(0),\tau_{y/\ep}\om)$.  But the problem is that the translation of this cube back to $Q_{1/\ep}(0)$ introduces the factor $\tau_{y/\ep}\om$ on the random parameter.  This is a priori not compatible with the Subadditive Theorem and requires more careful attention.  Nonetheless, the desired outcome can be reached, and it is the culmination of Lemmas \ref{lem:SubCubePositiveContact} and \ref{lem:mlGoodSetUniform}.

\begin{lem}\label{lem:mlGoodSetUniform}
Given any $\eta>0$, there exists a set, $G_\eta(l,\phi)$, such that $\P(G_\eta(l))>(1-\eta)$ and the convergence of $m^{\ep,l}(\om)$ is uniform for $\om\in G_\eta(l,\phi)$.
\end{lem}
\begin{proof}[Proof of Lemma \ref{lem:mlGoodSetUniform}]
This is simply Egorov's Theorem applied to the convergence from Lemma \ref{lem:SubAdditiveLimits}.
\end{proof}

\begin{proof}[Proof of Lemma \ref{lem:SubCubePositiveContact}]
Without loss of generality we assume that $r=1/N$ for some $N$.  We begin with a partition of $Q_1(0)$ into $N^n$ subcubes given by $\{Q_{r}(y_j)\}$ with appropriate $y_1,\dots,y_{N^n}$.  Let $\Om_\eta$ be the sets corresponding to Lemma \ref{lem:VaradhanGoodSetTranslation} applied to $G_\eta$ in Lemma \ref{lem:mlGoodSetUniform} and define the set of full measure,
\begin{equation*}
\Om_0=\intersect_{\eta\in \rational,\eta>0}\Om_\eta.
\end{equation*}

We now shift the original $\{y_j\}$ slightly to obtain a new collection, $\{\hat y^\ep_j\}$ so that we can make sure $\tau_{\hat y^\ep_j/\ep}\om\in G_\eta(l,\phi)$ even though $\tau_{y_j/\ep}\om$ may not be.  The family $\{\hat y^\ep_j\}$ is given by Lemma \ref{lem:VaradhanGoodSetTranslation} and we know that 
\begin{equation*}
\abs{y_j-\hat y^\ep_j}\leq \rho(\eta).
\end{equation*}
This gives that 
\begin{align}
&Q_1(0)\subset \union Q_{r+2\rho(\eta)}(\hat y^\ep_j)\\
\intertext{and}
&m^{\ep,l}( Q_{r+2\rho(\eta)}(\hat y^\ep_j),\om)=m^{\ep,l}( Q_{r+2\rho(\eta)}(0),\tau_{\hat y^\ep_j/\ep}\om)\to\bar m^l\ \text{uniformly in}\ \hat y^\ep_j
\end{align}

Now we collect some facts about $\abs{K^l(Q_1,\om)\intersect Q_{r}(\hat y^\ep_j)}$.  It will be easier to work with the $1/\ep$ scale picture, and so we are considering $Q_{1/\ep}$ and $Q_{r/\ep}$.  Because  
\begin{equation*}
Q_{1/\ep}\intersect Q_{(r+2\rho(\eta))/\ep}(\hat y^\ep_j/\ep)\subset Q_{(r+2\rho(\eta))/\ep}(\hat y^\ep_j/\ep),
\end{equation*}
the Monotonicity property (Lemma \ref{lem:ObstacleMonotonicity}) tells us that
\begin{align}
\abs{K^l(Q_{1/\ep},\om)\intersect Q_{(r+2\rho(\eta))/\ep}(\hat y^\ep_j/\ep)}\leq \abs{K^l(Q_{(r+2\rho(\eta))/\ep}(\hat y^\ep_j/\ep),\om)}.
\end{align}
Furthermore, we know that uniformly the limit holds:
\begin{align}
\frac{1}{\abs{Q_{(r+2\rho(\eta))/\ep}}}M^l(Q_{(r+2\rho(\eta))/\ep} + \hat y^\ep_j/\ep,\om)\to\al.
\end{align}
Thus for $\gam>0$ given, we choose $\ep$ small enough that
\begin{align}
&\frac{1}{\abs{Q_{(r+2\rho(\eta))/\ep}}}M^l(Q_{(r+2\rho(\eta))/\ep} + \hat y^\ep_j/\ep,\om)\leq (1+\gam)\al\\
\intertext{and}
&\frac{1}{\abs{Q_{1/\ep}}}M^l(Q_{1/\ep},\om)\geq (1-\gam)\al.
\end{align}

We can now estimate 
\begin{equation*}
\sum_j \abs{K(Q_{1/\ep})\intersect Q_{(r+2\rho(\eta))/\ep}(\hat y^\ep_j)}
\end{equation*}
from below and above as
\begin{align}
&(1-\gam)\abs{Q_{1/\ep}}\al \nonumber \\
&\leq M^l(Q_{1/\ep},\om)\nonumber\\
&\leq \sum_{j=1}^{1/r^n} \abs{K^l(Q_{1/\ep},\om)\intersect Q_{(r+2\rho(\eta))/\ep}(\hat y^\ep_j/\ep)}\label{eq:SubCubesSummation}\\
&\leq \sum_{j=1}^{1/r^n}\abs{K^l(Q_{(r+2\rho(\eta))/\ep}(\hat y^\ep_j/\ep),\om)}\nonumber\\
&= \sum_{j=1}^{1/r^n}M^l(Q_{(r+2\rho(\eta))/\ep},\tau_{\hat y^\ep_j/\ep}\om)\nonumber\\
&\leq \frac{1}{r^n}\abs{Q_{(r+2\rho(\eta))/\ep}}(1+\gam)\al\nonumber\\
&=\frac{(r+2\rho(\eta))^n}{\ep^nr^n}(1+\gam)\al.\nonumber
\end{align}
Because we have the upper bound on each of the sets,
\begin{align*}
\abs{K^l(Q_{1/\ep},\om)\intersect Q_{(r+2\rho(\eta))/\ep}(\hat y^\ep_j/\ep)} \leq \left((r+2\rho(\eta))/\ep\right)^n(1+\gam)\al,
\end{align*}
and taking $\ep,\gam,\eta$ all small enough, we conclude that all the terms in the summation of (\ref{eq:SubCubesSummation}) must be positive.  This concludes the lemma.
\end{proof}

\section{Solving The ``Corrector'' Equation}\label{sec:Corrector}
\setcounter{equation}{0}

Going back to the brief discussion of Section \ref{sec:BackgroundSubSecCorrector} and considering the results of Lemmas \ref{lem:ZeroContactLimit} and \ref{lem:PositiveContactLimit}, we now see that to balance the possible upper and lower limits of $w^{0,\ep}_l$, the good choice for $\Bar F$ will be the one such that it is at the boundary between the collection of $l$ giving a zero contact limit and the collection of $l$ giving a positive contact limit.  Indeed by the monotonicity of $\bar m^l$ in $l$, this is a reasonable choice.

\begin{DEF}\label{def:FbarDef}
The constant $\Bar F(\phi,0)$ is defined as 
\begin{equation}\label{eq:BarFAt0Def}
\Bar F(\phi,0):= \sup\{l: \bar m^l(\phi)=0\},
\end{equation}
and the constant $\Bar F(\phi,x_0)$ is defined as
\begin{equation}\label{eq:BarFAtGenericXDef}
\Bar F(\phi,x_0):= \Bar F(\phi(\cdot+x_0),0).
\end{equation}
\end{DEF}

\subsection{Solving The ``Corrector'' Equation At $x_0=0$}\label{sec:CorrectorAtX0}
Here we briefly comment on the proof of Proposition \ref{prop:Corrector} for the case that the cube is centered at $x_0=0$, $Q_1(0)$.  It is carried out almost exactly as in \cite{Schw-10Per}.  It consists of two lemmas which together yield Proposition \ref{prop:Corrector}, and the proofs of which are almost identical to \cite[Lemmas 3.5 and 3.7]{Schw-10Per}, so we omit them.  The main feature to note is that item (i) of the dichotomy in \cite[p.2661]{Schw-10Per} corresponds to $\bar m^l>0$, and item (ii) of the dichotomy corresponds to $\bar m^l=0$.  We simply note that it will be important to take as a definition of $\Om_\phi$,
\begin{equation*}
\Om_\phi := \intersect_{l\in\rational} \Om_{\phi,l}.
\end{equation*}
Once this is done, the proofs of the next two propositions go almost identically as to the proofs given in \cite{Schw-10Per}, with some very minor modifications to account for the fact that one must work with $l\in\rational$ and $\om\in\Om_\phi$.

\begin{lem}[Lemma 3.5 of \cite{Schw-10Per}]\label{lem:FBarCorrectDecay}
If $l=\Bar F(\phi,0)$ then for $\om\in\Om_\phi$ $w^{0,\ep}_l(\om)$ solving (\ref{eq:lRHS}) also satisfies $\norm{w^{0,\ep}_l}_{L^\infty}\to0$ as $\ep\to0$.
\end{lem}

\begin{lem}[Lemma 3.7 of \cite{Schw-10Per}]\label{lem:FBarUnique}
If $\om\in\Om_\phi$ and $l$ is any number such that $w^{0,\ep}_l$ solving (\ref{eq:lRHS}) satisfies $\norm{w^{0,\ep}_l}_{L^\infty}\to0$ as $\ep\to0$, then $l=\Bar F(\phi,0)$.
\end{lem}

\subsection{Solving The ``Corrector'' Equation At a Generic $x_0$}\label{sec:CorrectorAtGenericX}
\setcounter{equation}{0}

We now arrive at the proof of Proposition \ref{prop:Corrector} for a generic $x_0$.  This is where the random homogenization deviates slightly from the periodic case in the sense that the arguments applied in Section \ref{sec:SubadditiveAtX0} do not carry over directly to a generic $x_0\not=0$.  Instead, the information of the proof of Proposition \ref{prop:Corrector} must be obtained in a ``local uniform'' fashion near $x_0=0$, and then the ergodicity of the problem allows for the behavior of a generic $x_0$ to be captured by translating the equation to a point nearby $x_0=0$ and using the ``local uniform'' nature of the information there.  This argument is becoming a standard part of homogenization for nonlinear equations and can be seen explicitly used in: \cite[Lemmas 4.3, 4.4, Proof of Theorem 2.1]{KoReVa-06Homog}, \cite[Lemma 3.3, Theorem 3.5]{KoVa-06HomogTimeDep}, \cite[Lemmas 2.1, 2.2]{NoNo-11RandHomogG}, and \cite[Lemmas 5.1, 5.2, Proposition 5.3]{Schw-09}.

\begin{proof}[Proof of Proposition \ref{prop:Corrector} at $x_0\not=0$]

At $x_0=0$, we already identified the set $\Om_\phi$ where Proposition \ref{prop:Corrector} holds.  Now we appeal to Lemma \ref{lem:VaradhanGoodSetTranslation} in order to translate the behavior at a generic $x$ back to the origin.  Thus, we must take one more family of intersections to pick up the full measure sets from Lemma \ref{lem:VaradhanGoodSetTranslation}.  So for $\phi(\cdot+x_0)$, we let $\Om_{\phi(\cdot+x_0)}^\eta$ be the set obtained by applying Lemma \ref{lem:VaradhanGoodSetTranslation} to $\Om_{\phi(\cdot+x_0)}$, and so finally we set $\Tilde \Om_{\phi(\cdot+x_0)}$ as
\begin{equation*}
\Tilde \Om_{\phi(\cdot+x_0)}= \intersect_{\eta\in\rational\intersect(0,1)} \Om_{\phi(\cdot+x_0)}^\eta.
\end{equation*}
Let $R$ be fixed so that $x_0\in B_R(0)$.  Given any $\eta\in\rational\intersect(0,1)$, Lemma \ref{lem:VaradhanGoodSetTranslation} provides an $\hat x^\ep$ such that
\begin{equation*}
\abs{x-\hat x^\ep}\leq \rho(\eta)
\end{equation*}
and for all $\ep$ and $\om\in\tilde\Om_{\phi(\cdot+x_0)}$
\begin{equation*}
\tau_{\hat x^\ep/\ep}\om\in\Om_{\phi(\cdot+x_0)}.
\end{equation*}

Now we can use the uniform continuity of (\ref{eq:lRHS}) with respect to a change in the domain combined with the translation via $\hat x^\ep$ in order to use the result already established at $x_0=0$.   Let us record the proper auxiliary equation here:

\begin{equation}\label{eq:lRHSCorrectorSection}
\begin{cases}
\ds F_{\phi,x_0}(w^{x_0,\ep}_l,\frac{y}{\ep},\om) = l &\text{ in } Q_1(x_0)\\
w^{\ep}_l(y)=0 &\text{ on } \real^n\setminus Q_1(x_0).
\end{cases}
\end{equation}

We begin with the observation that if $w^{x_0,\ep}_{l}$ and $\hat w^\ep$ solve (\ref{eq:lRHSCorrectorSection}) with $l$ as a right hand side in $Q_1(x_0)$ and $Q_1(\hat x^\ep)$ respectively, then their difference is controlled by the uniform H\"older continuity (Theorem \ref{thm:CaSiUniformContinuityBoundary}) and the facts that they share the same boundary data and that for $\hat x^\ep$ small, the domains are very close.  Indeed since $w^{x_0,\ep}_{l}$ and $\hat w^{\ep}$ solve the same equation in $Q_{1}(x_0)\intersect Q_1(\hat x^\ep)$, 
\begin{align*}
\sup_{Q_1(x_0)\intersect Q_1(\hat x^\ep)}\abs{w^{x_0,\ep}_{l}-\hat w^{\ep}} 
&\leq \sup_{\real^n\setminus(Q_1(x_0)\intersect Q_1(\hat x^\ep))}\abs{w^{x_0,\ep}_{l}-\hat w^{\ep}}\\
&\leq C\abs{x_0-\hat x^\ep}^\gam.
\end{align*}
Hence also using again the regularity of $w^{0,\ep}_l$ and $\hat w^\ep$ and their respective boundary data, 
\begin{equation*}
\norm{ w^{x_0,\ep}_{l} - \hat w^{\ep}}_{L^\infty}\leq C\abs{x_0-\hat x^\ep}^\gam.
\end{equation*}

Next we can use the stationarity of the equations to move the equation for $\hat w$ in $Q_1(\hat x^\ep)$ back to the origin in $Q_1(0)$.  Therefore, we define
\begin{equation*}
\tilde w^\ep(x):=\hat w^{\ep}(x+\hat x^\ep).
\end{equation*}
The equation for $\tilde w^\ep$ (as the unique solution) is now 
\begin{equation*}
\begin{cases}
F_{\phi(\cdot+x_0),0}(\tilde w^{\ep},\frac{x}{\ep},\tau_{\hat x^\ep/\ep}\om)=l\ &\text{in}\ Q_1(0)\\
\tilde w^\ep=0\ &\text{on}\ \real^n\setminus Q_1(0).
\end{cases}
\end{equation*}
Because $\hat x^\ep$ is chosen so that $\tau_{\hat x^\ep/\ep}\om\in\Om_{\phi(\cdot+x_0)}$, we know by the part of Proposition \ref{prop:Corrector} already proved that $l=\Bar F(\phi(\cdot+x_0),0)$ gives the unique choice such that $\norm{\tilde w^\ep}_{L^\infty}\to0$.  Hence
\begin{equation*}
\limsup_{\ep\to0}\norm{w^{0,\ep}_l}_{L^\infty}\leq C(\rho(\eta))^\gam.
\end{equation*}

Thus letting $\eta\to0$ we see that for a.s.$\om$, $l=\Bar F(\phi(\cdot+x_0),0)$ is the unique choice of right hand side which gives the convergence of $\norm{w^{x_0,\ep}_l}\to0$.  Proposition \ref{prop:Corrector} is concluded by Definition \ref{def:FbarDef},
\begin{equation*}
\Bar F(\phi,x_0):=\Bar F(\phi(\cdot+x_0),0).
\end{equation*}

\end{proof}

\section{The Effective Operator and Proof of Theorems \ref{thm:main} and \ref{thm:general}}\label{sec:EffectiveAndProof}
\setcounter{equation}{0}
The ``corrector'' equation has been resolved, but before we can show the convergence of $u^\ep\to\bar u$, it still must be shown that the function $\Bar F(\phi,x_0))$ is an elliptic operator with respect to $M^-$ and $M^+$ and prove uniqueness for (\ref{eq:PIDEaveraged}).  This is basically the first half of Theorem \ref{thm:main}.  These properties appear as Proposition \ref{prop:PropertiesOfFbar} below and are proved almost exactly as in the periodic setting, \cite[Section 4]{Schw-10Per}.  So we state Proposition \ref{prop:PropertiesOfFbar} without proof and refer to the results of \cite[Lemmas 4.1, 4.2, Proposition 4.4]{Schw-10Per}.

\begin{prop}[Properties of $\Bar F$-- First Half of Theorem \ref{thm:main}]\label{prop:PropertiesOfFbar}
$\Bar F$ is a nonlocal elliptic operator in the sense of \cite[Definition 3.1]{CaSi-09RegularityIntegroDiff}, and hence has a comparison principle.  That is to say that for $u$ and $v$ both $C^{1,1}(x)$, then (for $M^-$, $M^+$ in either (\ref{eq:MminusDef}), (\ref{eq:MPlusDef}) or (\ref{eq:MminusSpecialA}), (\ref{eq:MplusSpecialA}))
\begin{equation*}
M^-(u-v,x)\leq \Bar F(u,x)-\Bar F(v,x)\leq M^+(u-v,x);
\end{equation*}
if $\phi\in C^{1,1}(D)\intersect L^\infty(\real^n)$ then $\Bar F(\phi,\cdot)\in C(D)$; and if $u$ and $v$ are respectively a usc subsolution and lsc supersolution of (\ref{eq:PIDEaveraged}), then
\begin{equation*}
\sup_{D}(u-v)\leq \sup_{\real^n\setminus D}(u-v).
\end{equation*}

\end{prop}

Before we prove the statement in Theorem \ref{thm:main} regarding convergence, we must identify the set $\Tilde\Om\subset\Om$.  In order to do so, we will need a special countable collection of smooth functions described below.  The existence of such a set is a straightforward exercise.

\begin{lem}\label{lem:CountablePhi}
There exists a countable family of smooth functions and dense points, $\{\phi_k\}_{k=1}^{\infty}$ and $\{x_k\}_{k=1}^\infty$, such that for any smooth $\phi$ and $x_0$, there are $\phi_k$ and $x_k$ with the following properties:
\begin{align}
&\phi-\phi_k \text{ has a maximum at } x_k, \label{eq:CountablePhiMax}\\
&x_k\to x_0, \\
&M^-(\phi-\phi_k)(x_k)\to0,\label{eq:CountablePhiMminusTo0}\\
&\norm{D^2\phi_k}_\infty \leq C(\phi), \text{ depending only on } \phi\\
&\phi_k\to\phi \text{ locally uniformly in D}\label{eq:CountablePhiLocalUniform}.
\end{align}
\end{lem}

\begin{rem}
We remark that there are many choices for the set of test functions in the definitions of viscosity solutions, which are all equivalent in this context.  The largest class would be those functions for which a subsolution, $u$, has a local maximum of $u-\phi$ over an open set, $N$, at $x_0$, and $\phi$ is only required to be punctually $C^{1,1}$ at $x_0$.  A smaller class would be those $\phi$ which are globally $C^2$ (or even smoother), and $u-\phi$ is required to attain a \emph{global} maximum at $x_0$.  Each class of test functions has its convenient time and place in these proofs, and that is why we have been vague with the presentation.  The interested reader can check the definitions and equivalence in \cite{BaIm-07} combined with \cite{CaSi-09RegularityIntegroDiff}.
\end{rem}

Let $\Bar F(\phi,x)$ be the nonlocal operator defined for a smooth $\phi$ by Definition \ref{def:FbarDef}, let $\Bar u$ be the solution of (\ref{eq:PIDEaveraged}), and let $u^\ep(\om)$ be the solution of (\ref{eq:PIDEmain}).  We will prove the second part of Theorem \ref{thm:main}, namely that for a.e. $\om$, $u^\ep\to\Bar u$ locally uniformly as $\ep\to0$.  Once the correct set $\tilde\Om\subset\Om$ is identified, the convergence is a straightforward application of the Perturbed Test Function Method of \cite[Section 3]{Evan-92PerHomog}.

\begin{proof}[Proof of Theorems \ref{thm:main} and \ref{thm:general}-- Convergence of $u^\ep\to\bar u$]
Here we are concerned with the convergence issue.  Define the set $\Tilde\Om$ using the countable dense class of test functions as
\begin{equation}
\Tilde\Om = \bigcap_{k=1}^\infty \Om_{\phi_k,x_k},
\end{equation}
where $\Om_{\phi_k,x_k}$ are give by Proposition \ref{prop:Corrector}.  The point being that for any $\om\in\Tilde\Om$, we know that Proposition \ref{prop:Corrector} holds simultaneously for all $\phi_k$ and $x_k$.  Thanks to the countability of $\phi_k,x_k$, we still have $P(\Tilde\Om)=1$.   We suppress the dependence on $\om$ for the remainder of the proof and therefore work with $u^\ep$ instead of $u^\ep(\om)$ for the remainder of the proof.

We only prove that $(u^\ep)^*$ is a subsolution of (\ref{eq:PIDEaveraged}).  The proof that $(u^\ep)_*$ is a supersolution follows similarly.  In what follows, we use one of the equivalent definitions of solutions of (\ref{eq:PIDEaveraged}) as given in \cite[Definition 1]{BaIm-07}.  This definition is equivalent to that of \cite[Definition 2.2]{CaSi-08A} under the assumption that the boundary data, $g$, is bounded and continuous on all of $\real^n\setminus D$.  Therefore, we must show that whenever $(u^\ep)^*-\phi$ has a \emph{strict} global maximum at $x_0$ for any smooth $\phi\in C^{1,1}(\real^n)\intersect L^{\infty}(\real^n)$, that the inequality holds:
\begin{equation*}
\Bar F(\phi,x_0)\geq0.
\end{equation*}

Proceeding by contradiction, suppose that $\phi$ is smooth and $u-\phi$ attains a strict global max at $x_0$ but the viscosity inequality fails:
\begin{equation*}
\Bar F(\phi,x_0)\leq -\del<0,
\end{equation*}
for some $\del>0$.  The goal will be to use Proposition \ref{prop:Corrector} to construct a local supersolution of (\ref{eq:PIDEmain}) near $x_0$ and contradict the strict maximum of $u-\phi$ at $x_0$.  First we must transfer the previous inequality to the functions in the countable class of Lemma \ref{lem:CountablePhi}.  Using the uniform ellipticity of $\Bar F$, the uniform continuity of $\Bar F$ (given by Lemma \ref{thm:CaSiUniformContinuityInterior}), and the uniform bound on $\norm{D^2\phi_k}$, we have
\begin{align*}
\Bar F(\phi,x_0)&= \Bar F(\phi,x_k)+\Bar F(\phi,x_0)-\Bar F(\phi,x_k)\\
&=\Bar F(\phi_k,x_k) + \Bar F(\phi,x_k) -\Bar F(\phi_k,x_k) + \Bar F(\phi,x_0)-\Bar F(\phi,x_k)\\
&\geq \Bar F(\phi_k,x_k) + M^-(\phi-\phi_k)(x_k) -\rho_\phi(\abs{x_k-x_0}).
\end{align*}
Thus for $k$ large enough by Lemma \ref{lem:CountablePhi}, we can make 
\begin{equation*}
\Bar F(\phi_k,x_k) \leq -\frac{\del}{2}.
\end{equation*}

Let $v^\ep$ be the solution of (\ref{eq:CorrectorEqDef}) in $Q_1(x_0)$ for $\Bar F(\phi_k,x_k)$.  We will now show that $\psi^\ep$ given by 
\begin{equation*}
\psi^\ep(y) = \phi_k(y) + v^\ep(y)
\end{equation*}
is in fact a supersolution of (\ref{eq:PIDEmain}) on an appropriately restricted ball, $B_R(x_k)$, for $R$ small enough.  We argue as though $v^\ep$ were a classical ($C^{1,1}$) solution.  This may not be the case, but converting the argument from the classical case to the viscosity solution case is by now standard (see \cite[Lemma 7.10]{Schw-10Per}).  

Indeed by Lemma \ref{lem:FPhiUnifContinuous}, we have for $y$ restricted to $B_R(x_0)$
\begin{equation*}
\abs{F(\phi_k + v^\ep,\frac{y}{\ep})- F_{\phi_k,x_k}(v^\ep,\frac{y}{\ep})}=
\abs{F_{\phi_k,y}(v^\ep,\frac{y}{\ep})- F_{\phi_k,x_k}(v^\ep,\frac{y}{\ep})}
\leq \rho_\phi(R),
\end{equation*}
and this holds anytime $v^\ep$ is $C^{1,1}$, but is \textbf{independent of the function} $v^\ep$ and $y$.   Thus restricting $R$ small enough so that $\rho_\phi(R)-\del/2\leq0$, we conclude that 
\begin{equation*}
F(\phi_k+v^\ep,\frac{y}{\ep})\leq 0 \ \text{ in }\ B_R(x_k).
\end{equation*}
Applying the comparison theorem, we see that for each $\ep$,
\begin{equation*}
\sup_{B_R(x_k)}(u^\ep-\phi_k-v^\ep)\leq \sup_{\real^n\setminus B_R(x_k)}(u^\ep-\phi_k-v^\ep).
\end{equation*}
First we keep $k$ fixed and take upper limits as $\ep\to0$.  Proposition \ref{prop:Corrector} implies the $v^\ep$ term vanishes and we obtain
\begin{equation*}
\sup_{B_R(x_k)}((u^\ep)^*-\phi_k)\leq \sup_{\real^n\setminus B_R(x_k)}((u^\ep)^*-\phi_k).
\end{equation*}
Finally using the uniform continuity of $u^\ep$ (hence $(u^\ep)^*$) and the uniformity of $\phi_k\to\phi$, we conclude 
\begin{equation*}
\sup_{B_R(x_0)}((u^\ep)^*-\phi)\leq \sup_{\real^n\setminus B_R(x_0)}((u^\ep)^*-\phi).
\end{equation*}
This contradicts the fact that the maximum of $u-\phi$ at $x_0$ was strict, and so we must have $\Bar F(\phi,x_0)\geq0$.

The proof that $(u^\ep)_*$ is a supersolution of (\ref{eq:PIDEaveraged}) follows analogously.  It is worth pointing out that due to the uniform continuity estimates on $u^\ep$ that are independent of $\ep$, both $(u^\ep)^*$ and $(u^\ep)_*$ are equal to $g$ on $\real^n\setminus D$.  Thus since (\ref{eq:PIDEaveraged}) has comparison, $(u^\ep)^*$, $(u^\ep)_*$, and $\bar u$ attain the same boundary data, and using that $\bar u$ is a solution, we conclude that 
\begin{equation*}
\bar u\leq (u^\ep)_*\leq(u^\ep)^*\leq \bar u.
\end{equation*}
This implies local uniform convergence to $\bar u$.
\end{proof}

\begin{rem}[Accounting of $\Tilde\Omega$]\label{rem:OmegaTilde}
First, $\phi$, and $l$ are fixed, and this gives a unique solution of the obstacle problems (\ref{eq:ObstacleBigDomain}), (\ref{eq:ObstacleSmallDomain}).  Then full measure sets which arise from the Subadditive Ergodic Theorem listed in Lemma \ref{lem:SubAdditiveLimits} are generated for a countable, dense family of $l$.  Lemma \ref{lem:VaradhanGoodSetTranslation} is used for a countable family of tolerances, $\eta$, in conjunction with Egoroff's Theorem, which is relevant for Lemma \ref{lem:PositiveContactLimit} via Lemma \ref{lem:SubCubePositiveContact}.  Finally all of the preceding sets are intersected along a countable family of $\eta$ and $l$, as well as the particular set from Lemma \ref{lem:SubAdditiveLimits} applied to the actual $\Bar F(\phi,x_0)$.  In order that this set will work at all $x_0$, Lemma \ref{lem:VaradhanGoodSetTranslation} is applied once more to account for translating the equation from $x_0$ back to $x=0$ (in Section \ref{sec:CorrectorAtGenericX}).  Another countable intersection is performed for these sets as well.  This gives rise to the set of full measure which is referred to as $\Omega_{\phi}$ appearing in the statement of Proposition \ref{prop:Corrector}.  At this point, the full measure set still depends on the test function, $\phi$, and so in the conclusion of Theorems \ref{thm:main} and \ref{thm:general}, there is one more countable intersection taken over the countable ``dense'' family of test functions mentioned in \ref{lem:CountablePhi}.  Then one has reached the actual set, $\Tilde \Omega$. 
\end{rem}

\section{Appendix-- Useful Facts}\label{sec:Appendix}
\setcounter{equation}{0}

In this section we briefly collect some useful facts used in this note.  Most items will be listed without proof, but specific references will be provided.

\subsection{Basic Properties of Equations (\ref{eq:PIDEmain}), (\ref{eq:PIDEaveraged}), and (\ref{eq:lRHS})}\label{sec:AppendixBasicProp}
Some of the major aspects of solutions to integro-differential equations used in this note are the Aleksandrov-Bakelman-Pucci estimate, regularity, and uniqueness.  Here we collect these results and state them without proof.  

\begin{thm}[ABP Type Estimate, Theorem 9.1 of \cite{GuSc-12ABParma}]\label{thm:ABP}
Suppose that $M^+_A$ is as in (\ref{eq:MplusSpecialA}) and
\begin{equation*}
\begin{cases}
\ds M^+_A(v,x)\geq -g(x) &\text{in}\ B\\
v\leq0 &\text{on}\ \real^n\setminus B.
\end{cases}
\end{equation*}
Then
\begin{equation*}
\sup_B\{v\}\leq \frac{C(n)}{\lam}\diam(B)(\norm{g}_{L^\infty})^{(2-\sig)/2}(\norm{g}_{L^n})^{\sig/2}.
\end{equation*}
\end{thm}

\begin{thm}[Interior H\"older Regularity, Theorem 12.1 of \cite{CaSi-09RegularityIntegroDiff}]\label{thm:CaSiUniformContinuityInterior}
If $u$ is bounded on $\real^n$ and is simultaneously a subsolution and a supersolution in $B_1$ of respectively
\begin{align*}
 M^-u\leq C_0\ \ \text{and}\ \ M^+u&\geq -C_0  \ \ \text{in}\ B_1,
\end{align*}
then $u$ is uniformly $\gam$- H\"older continuous in $B_{1/2}$ with $\gam$ depending only on the dimension, a lower bound on $\sig$, and ellipticity:
\begin{equation*}
[u]_{C^\gam(B_{1/2})} \leq C(\sup_{\real^n}\{u\}+C_0).
\end{equation*}
\end{thm}

\begin{thm}[Boundary Regularity, Theorem 3.3 of \cite{CaSi-09RegularityByApproximation}]\label{thm:CaSiUniformContinuityBoundary}
Assume (\ref{eq:AssumptionEllipticityPointwise}).  Then the solutions of (\ref{eq:PIDEmain}), (\ref{eq:PIDEaveraged}), and (\ref{eq:lRHS}) are uniformly continuous with a modulus that only depends on $\lam$, $\Lam$, $\sig$, $n$, the domain, the boundary data, and $\norm{f^{\al\beta}}_\infty$.  Moreover if the boundary data is H\"older continuous, then so is the solution with a possibly different H\"older exponent.  
\end{thm}

\begin{rem}
It is worth remarking that Theorems \ref{thm:CaSiUniformContinuityInterior} and \ref{thm:CaSiUniformContinuityBoundary} both carry over to the setting of (\ref{eq:AssumptionKernelsGS}), with extremal operators (\ref{eq:MminusSpecialA}), (\ref{eq:MplusSpecialA}).  The main details are discussed in \cite[Section 10]{GuSc-12ABParma}.
\end{rem}

\begin{lem}\label{lem:FPhiUnifContinuous}
The operator $F_{\phi,x}(v,y)$ is uniformly continuous in $x$, independent of $v$ and $y$.  That is there exists a modulus, $\rho_\phi$, depending only on $\phi$, such that if $v$ and $y$ are any function and any point for which $F(v,y)$ is well defined, then for any $x_1,x_2\in\real^n$, 
\begin{equation*}
\abs{F_{\phi,x_1}(v,y) - F_{\phi,x_2}(v,y)} \leq \rho_\phi(\abs{x_1-x_2}),
\end{equation*}
independent of $v$ and $y$.
\end{lem}

\begin{proof}[Sketch of Proof of Lemma \ref{lem:FPhiUnifContinuous}]
The proof of this lemma is a direct application of the results found in \cite[Lemma 4.2]{CaSi-09RegularityIntegroDiff}, specifically it follows from the assertion that $L^{\al\beta}\phi(\cdot)$ is uniformly continuous, uniformly in $\al$ and $\beta$. 
We recall 
\begin{equation}\label{eq:FPhiXnotNew}
F_{\phi,x_0}(v,y)= \inf_{\al}\sup_{\beta} \big\{ f^{\al\beta}(y) + [L^{\al\beta}\phi(x_0)](y) + L^{\al\beta}v(y)\big\},
\end{equation}
where $[L^{\al\beta}\phi(x_0)](y)$ is defined in (\ref{eq:FrozenOperatorDefLinear}).  Thanks to the bounds on $K^{\al\beta}$ from (\ref{eq:AssumptionEllipticityPointwise}) and the result \cite[Lemma 4.2]{CaSi-09RegularityIntegroDiff}, we know that for y fixed, $[L^{\al\beta}\phi(x_0)](y)$ is a uniformly equicontinuous (in $x_0$) family in $\al$ and $\beta$.  Therefore, for each $\al,\beta$,
\begin{align*}
&\left(f^{\al\beta}(y) + [L^{\al\beta}\phi(x_1)](y) + L^{\al\beta}v(y)\right)
-\left(f^{\al\beta}(y) + [L^{\al\beta}\phi(x_2)](y) + L^{\al\beta}v(y)\right)\\
&= [L^{\al\beta}\phi(x_1)](y)-[L^{\al\beta}\phi(x_2)](y)\\
&\leq \rho_\phi(\abs{x_1-x_2}),
\end{align*}
for some modulus, $\rho_\phi$.  Due to the uniformity in $\al$ and $\beta$ the result holds under operations of taking the infimum and supremum, and hence for $F_{\phi,x_0}$.
\end{proof}

\begin{lem}[Comparison for (\ref{eq:PIDEaveraged})]\label{lem:ComparisonForFbar}
Given uniformly continuous boundary data, $g$, the equation (\ref{eq:PIDEaveraged}) has a unique solution.
\end{lem}

\begin{proof}[Proof of Lemma \ref{lem:ComparisonForFbar}]
This is a direct application of the comparison results of \cite[Section 5]{CaSi-09RegularityIntegroDiff} once it has been established that $\Bar F$ is indeed elliptic with respect to $M^-$ and $M^+$.
\end{proof}

\subsection{Pushing The Subadditive Limit Around By $\tau_{x/\ep}\om$}\label{sec:AppendixPushLimitAround}
The following is an incredibly useful lemma for nonlinear stochastic homogenization.  It seems to have been first used for this purpose in \cite[Proof of Theorem 2.1]{KoReVa-06Homog}.  The need for such a result stems from the fact that it is natural to prove results at different spacial locations by using the translations $u^\ep(\cdot+x,\om)=u^\ep(\cdot,\tau_{x/\ep}\om)$.  However, one must be careful that $\tau_{x/\ep}\om$ is still in an appropriate subset of $\Om$.  We state the results here and copy the proof presented in \cite[Lemma 5.7]{Schw-09}.

\begin{lem}[Kosygina-Rezakhanlou-Varadhan]\label{lem:VaradhanGoodSetTranslation}
Let $G_\eta$ be such that $\P(G_\eta)\to1$ as $\eta\to0$.  Then there exists a function, $\rho(\eta)$, and a set of full measure, $\Om_\eta$, such that for $\ep$ chosen small enough:
\begin{align}
&\rho(\eta)\to0\ \text{as}\ \eta\to0,\\
&\forall \om\in\Om_\eta\ \text{and}\ \forall x\in B_{1/\ep}(0),\ \text{there is}\ \hat x\  \text{such that}\ 
&\hat x\in \{x:\tau_{x}\om\in G_\eta\}\intersect B_{1/\ep}(0),\\
&\text{and}\ \abs{x-\hat x}\leq \frac{\rho(\eta)}{\ep}.
\end{align}
\end{lem}

\begin{proof}[Proof of Lemma \ref{lem:VaradhanGoodSetTranslation}]
This proof will be a consequence of the Ergodic Theorem combined with the regularity of Lebesgue measure on $\real^n$. 

We begin by applying the ergodic theorem to the function $F_\eta$, defined as:
\begin{equation*}
F_\eta(\om) = 
\begin{cases}
1 &\text{ if } \om\in G_\eta \\
0 &\text{ otherwise }.
\end{cases}
\end{equation*}
The ergodic theorem says there exists $\Om_\eta$ with $P(\Om_\eta)=1$ and $\forall\om\in\Om_\eta$
\begin{equation*}
\lim_{r\to \infty} \frac{1}{\abs{B_r}} \int_{B_r} F_\eta(\tau_{x,s}\om)dxds = \int_{\Om}F_\eta(\om)d\P(\om).
\end{equation*}
Specifically, for $\ep$ small enough and $\forall\om\in\Om_\eta$:
\begin{equation*}
\frac{ \abs{ \{x:\tau_{x}\om\in G_\eta \}\cap B_{1/\ep} } }{ \abs{ B_{1/\ep} } } \geq P(G_\eta)-\eta
\geq 1-2\eta.
\end{equation*}
In other words, 
\begin{equation*}
\abs{ \{x:\tau_{x}\om\in G_\eta \}\cap B_{1/\ep} } \geq (1-2\eta)\abs{ B_{1/\ep}}.  
\end{equation*}

In order to find the function $m(\eta)$, we will use the regularity property of Lebesgue measure.  Let us call the good set $G = \{x:\tau_{x}\om\in G_\eta \}\cap B_{1/\ep}$ and the bad set will be $G^c\cap B_{1/\ep}$.  The outer regularity of Lebesgue measure says that there is a basic set (a finite union of balls), 
\begin{equation*}
E=\bigcup_{i=1}^M B_{r_i},
\end{equation*}
such that $G^c\cap B_{1/\ep}\subset E$ and  
\begin{equation*}
\abs{E}-\eta \leq \abs{G^c\cap B_{1/\ep}}.
\end{equation*}
We also know from above that 
\begin{equation*}
\abs{G^c\cap B_{1/\ep}}=\abs{B_{1/\ep}}-\abs{G}\leq 2\eta\abs{B_{1/\ep}}.
\end{equation*}
Hence
\begin{equation*}
\abs{E}\leq 2\eta\abs{B_{1/\ep}} + \eta.
\end{equation*}
The worst case scenario regarding the distance from $x\in E$ to $\hat x \in E^c\cap B_{1/\ep}$ is when $E$ is one ball, and $x$ is at its center.  Thus 
\begin{align*}
&\abs{x-\hat x}\leq \bigl(\frac{2\eta C_{n}}{\ep^{n}}+\eta\bigr)^{1/(n)}\\
&=\bigl(\frac{(2C_{n}\eta+\ep^{n}\eta)}{\ep^{n}}\bigr)^{1/(n)} \leq \frac{(3C_{n}\eta)^{1/(n)}}{\ep}\\
&:=\frac{m(\eta)}{\ep}. 
\end{align*}
Which completes the proof of the lemma.
\end{proof}

\subsection{Facts About The Obstacle Problem}\label{sec:AppendixObstacleProb}
Finally, we collect a few facts about the obstacle problem which are useful above.  The first is the representation of the least supersolution (\ref{eq:ObstacleGeneric}) as the solution of a variational inequality:

\begin{equation}\label{eq:VariationalInequality}
\begin{cases}
\max(F_{\phi,x_0}(U^l_A,x)-l,0-U^l_A)=0\  &\text{in}\ A\\
U^l_A=0\ &\text{on}\ \real^n\setminus A.
\end{cases}
\end{equation}
Furthermore, by comparison between $U^l_A$ and the obstacle at points on the contact set, we know that for all $x\in K(A)$
\begin{equation}\label{eq:ObstacleEquationContactRHS}
F_{\phi,x_0}(U^l_A,x)\geq F_{\phi,x_0}(0,x),
\end{equation}
where on the right hand side of the inequality the operator is applied to the constant, $0$, function.

The first result we mention about the obstacle problem is its a priori uniform continuity.  This follows as a direct analog to the nonlocal setting of the proof provided in \cite[Theorem 2.1 (i)]{CaSoWa-05} using the penalization method of approximating the obstacle solution and so we omit the proof.

\begin{lem}[Regularity For The Obstacle Problem]\label{lem:ObstacleUniformModulus}
There exists an exponent, $\gam$, depending only on $\lam$, $\Lam$, $\sigma$, $n$, and $A$, such that $U^l_A$ is H\"older continuous with exponent $\gam$.
\end{lem}

\begin{lem}[Translation]\label{lem:ObstacleTranslation}
The obstacle solution, $U^l_A$, satisfies the translation property: for any $z\in\real^n$, on the set $A+z$,
\begin{equation*}
U^l_{A+z}(\cdot,\om) = U^l_{A}(\cdot - z,\tau_z\om).
\end{equation*}
\end{lem}

\begin{proof}[Proof of Lemma \ref{lem:ObstacleTranslation}]
Without loss of generality we take $F_{\phi,x_0}$ to simply be $F$ (the argument does not see any dependence on $[L^{\al\beta}\phi(x_0)](x)$).  We will show that $U^l_{A+z}(\cdot,\om) \leq U^l_{A}(\cdot - z,\tau_z\om)$, and the reverse inequality follows similarly.  The equation for $U^l_A$ and the stationarity of $F$ implies that for $y\in A+z$, hence $y=x+z$ for some $x\in A$, we have
\begin{equation*}
F(U^l_A(\cdot-z,\tau_z\om),y,\om)=F(U^l_A(\cdot-z,\tau_z\om),x+z,\om) = F(U^l_A(\cdot,\tau_z\om),x,\tau_{z}\om).
\end{equation*}
Hence from the equation for $U^l_A$, we have
\begin{equation}
F(U^l_A(\cdot-z,\tau_z\om),y,\om)\leq l.
\end{equation}
Moreover, $U^l_A(\cdot-z,\tau_z\om)\geq 0$ in $\real^n$.  Since $U^l_{A+z}$ is the least such supersolution, we conclude
\begin{equation*}
U^l_{A+z}(\cdot,\om)\leq U^l_A(\cdot-z,\tau_z\om).
\end{equation*}
\end{proof}

\begin{lem}[Monotonicity In The Domain]\label{lem:ObstacleMonotonicity}
If $A\subset B$, then $U^l_A\leq U^l_B$.
\end{lem}

\begin{proof}[Proof of Lemma \ref{lem:ObstacleMonotonicity}]
As in the previous proof, we work without loss of generality with $F$ instead of $F_{\phi,x_0}$.  Since $A\subset B$, $F(U^l_B,x)\leq l$ in $A$ and $U^l_B\geq0$ in $\real^n$.  Since $U^l_A$ is the least such supersolution, we conclude $U^l_A\leq U^l_B$.
\end{proof}

\begin{lem}[Monotonicity In The RHS]\label{lem:ObstacleMonotonicityInRHS}
If $l_1\leq l_2$, then $U^{l_1}_A\geq U^{l_2}_A$.
\end{lem}

\begin{proof}[Proof of Lemma \ref{lem:ObstacleMonotonicityInRHS}]
We notice that for any $u$ solving $F_{\phi,x_0}(u,x)\leq l_1$, then also $u$ solves $F_{\phi,x_0}(u,x)\leq l_2$.  Hence taking the infimum over all such supersolutions, we find
\begin{equation*}
U^{l_2}_A\leq U^{l_1}_A.
\end{equation*}

\end{proof}


\bibliography{../refs}
\bibliographystyle{plain}
\end{document}